\documentclass[11pt,a4paper]{amsart}
\usepackage{amssymb,amscd}
\usepackage{mathrsfs}
\usepackage[mathcal]{eucal}
\usepackage{float}
\usepackage[stable]{footmisc}

\usepackage{xy}
\input xy
\xyoption{all}
\usepackage{pdflscape}
\usepackage{hyperref}

\usepackage[english]{babel}

\def\Q{\mathbb{Q}}
\def\R{\mathbb{R}}
\def\C{\mathbb{C}}

\def\Hom{\mathrm{Hom}}

\def\GL{\mathrm{GL}}
\def \cs {\mathfrak{s}}
\def \sort {\mathcal{S}}

\def\cc{\mathrm{cc}}

\let\myacute=\'

\def\lri{\mathcal{O}}

\def\isom{\cong}

\def\ra{\longrightarrow}

\def\<{\langle}
\def\>{\rangle}
\def\N{\mathbb{N}}
\def\Z{\mathbb{Z}}

\def\lfi{K}
\def\lri{\mathfrak{o}}
\def\gri{\mathcal{O}}
\def\maxideal{\mathfrak{p}}

\def\resfield{\mathfrak{K}}

\def\Mat{\mathrm{Mat}}
\def\ac{\overline{\mathrm{ac}}}
\def \Qp {\mathbb{Q}_p}

\def\vf{\mathrm{Denef}\text{-}\mathrm{Pas}}
\def\vfp{\mathrm{Denef}\text{-}\mathrm{Pas},\mathrm{P}}

\def\hac{\mathrm{H,ac}}

\newcommand{\kGq}[1]{k_{#1}(q)}
\newcommand{\tqalpha}[1]{k^{#1}(q)}

\newtheorem{thmABC}{Theorem}

\newenvironment{Questionnn}{\noindent{\bf Question.}
  \nonumber \it}{\bigskip}

\def \begindm {\begin{displaymath}}

\def \enddm {\end{displaymath}}
\def \naturals {\mathbb{N}}
\def \stab {\ensuremath{\mathrm{Stab}}}
\def \suchthat {\ensuremath{:}}

\def\C{\mathbb{C}}
\def\Q{\mathbb{Q}}

\def\cL{\mathcal{L}}
\def\cT{\mathcal{T}}

\newtheorem{theo}{Theorem}[section]
\newtheorem{lemma}[theo]{Lemma}
\newtheorem{cor}[theo]{Corollary}
\newtheorem{prop}[theo]{Proposition}

\theoremstyle{definition}
\newtheorem{Def}[theo]{Definition}

\numberwithin{equation}{section}

\long\def\symbolfootnote[#1]#2{\begingroup%
\def\thefootnote{\fnsymbol{footnote}}\footnote[#1]{#2}\endgroup}

\title{Uniform cell decomposition with applications to Chevalley groups}
\author{Mark N. Berman${}^{\S}$, Jamshid Derakhshan${}^\dag$, Uri Onn${}^{*}$ and Pirita Paajanen${}^{ \ddag}$}

\thanks{
${}^\S$ Supported by the Skirball Foundation via the Center for Advanced Studies in Mathematics at Ben-Gurion University of the Negev;
${}^\dag$Supported by the Center for Advanced Studies in Mathematics at Ben Gurion University;
${}^{*}$ Supported by Israel Science Foundation (ISF);
${}^\ddag$ Supported by an Academy of Finland Postdoctoral Fellowship.}

\address{Department of Mathematics and Applied Mathematics, University of Cape Town,
Rondebosch 7701 South Africa}
\email{mark.berman@uct.ac.za}

\address{Mathematical Institute, University of Oxford,
24-29 St. Giles, Oxford OX1 3LB, UK}
\email{derakhsh@maths.ox.ac.uk}

\address{Department of Mathematics, Ben Gurion
  University of the Negev, Beer-Sheva 84105 Israel}
\email{urionn@math.bgu.ac.il}

\address{Department of Mathematics and Statistics, University of Helsinki, P.O. Box 68, FI-00014 University of Helsinki, Finland}
\email{pirita.paajanen@helsinki.fi}

\begin{document}

\keywords{Chevalley groups, Conjugacy classes, $\mathfrak{p}$-adic
integration, Denef-Pas language, Cell decomposition, Quantifier
elimination}

\subjclass[2000]{20G25, 03C10, 03C98, 11M41}

\begin{abstract} We express integrals of definable functions over definable sets uniformly
for non-Archimedean local fields, extending results of Pas. We
apply this to \nobreak{Chevalley} groups, in particular proving
that zeta functions counting conjugacy classes in congruence
quotients of such groups depend only on the size of the residue
field, for sufficiently large residue characteristic. In
particular, the number of conjugacy classes in a congruence
quotient depends only on the size of the residue field. The same
holds for zeta functions counting dimensions of Hecke modules of
intertwining operators associated to induced representations of
such quotients.
\end{abstract}

\maketitle

\section{Introduction}\label{section-introduction}

We provide a framework for studying integrals over Chevalley
groups which makes them amenable to model-theoretic tools. This
has various applications to counting problems associated with
Chevalley groups. Let $G$ be a Chevalley group defined over $\Z$ and
let $\lfi$ be a non-Archimedean local field with valuation ring
$\lri$, maximal ideal $\maxideal$, absolute value $|~|$ and
residue field $\resfield$ of size $q$. Suppose we have a counting
problem related to $G(\lri)$, given by a sequence of natural
numbers $A=(a_m)$, which we encode in a Poincar\myacute{e} series
$\zeta_A(s)=\sum_{m=0}^{\infty} a_m q^{-ms}$. In the cases of
interest to us, this generating function is expressible by an
integral $Z_A(s)=\int_{G(\lri)}|f(g)|^s d\mu_{G(K)}(g)$, where $f$
is a function which is definable in the Denef-Pas language for
Henselian valued fields (see Section \ref{MT1}), and $\mu_{G(K)}$
is the normalised Haar measure on $G(K)$. Following the Ax-Kochen
philosophy we show that the integral $Z_A(s)$ is unchanged if we
replace the field $K$ by a field $K'$ with the same residue field,
provided the residue characteristic is sufficiently large. This is
reflected in the corresponding counting problems under
consideration. The following theorem allows us to transform the
problem into a model-theoretic framework.

\begin{thmABC}\label{theorem-coord} Let $G$ be a Chevalley group defined over $\Z$. There exist (explicitly computed)
rational maps $\iota: \mathbb{A}^{\dim G} \to G$ and $\alpha:
\mathbb{A}^{\dim G} \to \mathbb{A}$ such that the following holds.
Let $K$ be a non-Archimedean local field  with valuation ring
$\lri$. Let $dx$ denote the additive Haar measure on $K^{\dim G}$
normalized on $\lri^{\dim G}$. Then the measure $\mu$ given by
$$\int_{G(K)} f(g) d\mu:=\int_{K^{\dim G}} |\alpha(x)| f(\iota(x)) dx,$$
as $f$ runs through all complex valued Borel functions on $G(K)$,
is a left and right Haar measure on $G(K)$ normalized on
$G(\lri)$.
\end{thmABC}

Theorem~\ref{theorem-coord} shifts the focus to integrals whose
domain of integration and integrand are definable in a suitable language.
We use
$\cL_{\vf}$, the Denef-Pas language for Henselian valued fields
(see \cite{Pas}), which entails a valued field sort, a residue
field sort and an ordered abelian group sort (the latter with top
element), equipped with function symbols interpreted as valuation
and angular component. For more details, see Section~\ref{MT1}
below. The following theorem gives a uniform formula for such
integrals.

\begin{thmABC}\label{theorem-model-theory} Let $x$ be an $m$-tuple of valued field variables.
Let $\psi(x)$ be an $\cL_{\vf}$-formula and let $\Phi(x)$ be an
$\cL_{\vf}$-definable function into the valued field. Assume there
is some $\sigma_0\in \R$ such that the integral
$\int_{W_K}|\Phi(x)|^s dx$ converges for all $s\in \C$ with
$\text{Re}(s) > \sigma_0$ for all non-Archimedean local fields $K$
of sufficiently large residue characteristic, where $W_K=\{x \in
K^m \suchthat K\models \psi(x) \}$. Then there exist a constant
$N$, rational functions $R_1(X,Y),\dots,R_k(X,Y)$ over $\Z$, and
formulae $\psi_1(x),\dots,\psi_k(x)$ in the language of rings such
that for any non-Archimedean local field $K$ with residue field
$\resfield$ of characteristic at least $N$, and for all $s\in \C$
with $\text{Re}(s) > \sigma_0$, we have
\begin{equation}\label{eqn.theorem.B}
\int_{W_K} |\Phi(x)|^s
dx=m_1(\resfield)R_1(q,q^{-s})+\dots+m_k(\resfield)R_k(q,q^{-s}),
\end{equation}
where $m_i(\resfield)$ denotes the cardinality of the set of
$\resfield$-points of $\psi_i(x)$, $q$ is the cardinality of the
residue field $\resfield$, and $dx$ denotes the Haar measure on
$K^m$ normalized such that $\lri^m$ has measure $1$.
\end{thmABC}

In model theoretic language, the sum in (\ref{eqn.theorem.B}) is
known as a uniform expression for the integral. We shall deduce a
similar result for integrals of the form $\int_{W_K}
q^{-\Theta(x)s} dx$ where $\Theta(x)$ is a definable function from
$K^m$ into the value group $\Z$ (Corollary \ref{cor-of-B}). 
Theorem~\ref{theorem-model-theory} is a transfer principle for values of the
integrals stating that the value depends only on the size of the residue field. For the fields $\Qp$, it follows from the
proof of Theorem 5.1 in Pas \cite{Pas}, for almost all $p$. To prove Theorem~\ref{theorem-model-theory} 
we use Pas' methods and generalize his results to non-Archimedean local
fields. 

The use of model theory to evaluate $\maxideal$-adic integrals of
the form $\int_{\psi(K)} |f(x)|^s dx$, where $\psi(K)$ denotes the
definable subset of $K^m$ for a $\maxideal$-adic field $K$ defined
by a formula $\psi(x)$ of the language of valued fields, and
$f(x)$ denotes a definable function, originated in the work of
Denef. Theorem~\ref{theorem-model-theory} can also be deduced from the 
motivic rationality theorem of 
Cluckers-Loeser \cite[Theorem~14.4.1]{CL1} combined with their specialization principle 
\cite[Theorem~9.1.5]{CL1}. An extension to the wider setting of constructible 
exponential motivic functions can be deduced from the Cluckers-Loeser motivic transfer principle 
\cite[Theorem~9.2.4]{CL2}.
It is also possible to deduce Theorem~\ref{theorem-model-theory} from the work of Denef-Loeser, via the 
techniques in the proof of \cite[Theorem 8.5.1]{DL}. A uniform expression for similar integrals is
proved by Macintyre in \cite{Macintyre}. 

% We note that the use of model theory to evaluate $\maxideal$-adic integrals of
% the form $\int_{\psi(K)} |f(x)|^s dx$, where $\psi(K)$ denotes the
% definable subset of $K^m$ for a $\maxideal$-adic field $K$ defined
% by a formula $\psi(x)$ of the language of valued fields, and
% $f(x)$ denotes a definable function, originated in the work of
% Denef.
Our first application of Theorem~\ref{theorem-model-theory}
concerns a zeta function counting conjugacy classes.  Given a
Chevalley group $G$ with an embedding into $\GL_d$, consider the
congruence subgroups
$$G^m(\lri):=\mathrm{Ker}(G(\lri)\to \mathrm{GL}_d(\lri/\maxideal^m))$$
of $G(\lri)$. For each $m \in \N$, let $c_m$ denote the the number
of conjugacy classes in the congruence quotient $G(\lri,m):=
G(\lri)/G^m(\lri)\cong G(\lri/\maxideal^m)$. Following du Sautoy
\cite{dScc}, we define the conjugacy class zeta function of
$G(\lri)$ by
$$\zeta^{\cc}_{G(\lri)}(s):=\sum_{m=0}^\infty c_m
q^{-ms}.$$

\begin{thmABC}\label{theorem-cc}
Let $G$ be a Chevalley group. There exist formulae
$\psi_1(x),\dots,\psi_{k}(x)$ in the language of rings, rational
functions  $R_1(X,Y)$,...,$R_{k}(X,Y)$ over $\Z$ and a constant
$N$, such that for all complete discrete valuation rings $\lri$
with finite residue field $\resfield$ of characteristic $p$ and
cardinality $q$, where $p>N$, and for all $s\in \C$ with
$\text{Re}(s)> \sigma_0$ for some $\sigma_0\in \R$, the conjugacy
class zeta function depends only on $q$ and can be written as
$$\zeta^{\cc}_{G(\lri)}(s)=m_1(\resfield)R_1(q,q^{-s})+\dots+m_k(\resfield)R_k(q,q^{-s}),$$
where $m_i(\resfield)$ denotes the cardinality of the set of
$\resfield$-points of $\psi_i(x)$, for $1\leq i \leq k$. In
particular, if $\lri$ and $\lri'$ are complete discrete valuation
rings with the same finite residue field of characteristic larger
than $N$, then $G(\lri)$ and $G(\lri')$ have the same conjugacy
zeta function. Consequently, the number of conjugacy classes in
the groups $G(\lri,m)$ and $G(\lri',m)$ is the same for all $m$.
\end{thmABC}

The rationality of the conjugacy zeta function
$\zeta^{\cc}_{G(\Z_p)}(s)$ was proved by du Sautoy \cite{dScc} for
$G$ a compact $p$-adic analytic group. The novelty of
Theorem~\ref{theorem-cc} is that it holds also in positive
characteristic and moreover that the resulting zeta function
depends only on the size of the residue field.

Our second application is of interest in the study of Hecke
modules of intertwining operators of induced complex representations of
Chevalley groups, which are described by means of suitable double
coset spaces. Let $H$ be a finite group and let $Q_1,Q_2$ be
subgroups of $H$. We write $\rho_i=\mathbf{1}_{Q_i}^H$ for the
representation obtained by inducing the trivial representation
of $Q_i$ to $H$. The space $\Hom_H(\rho_1,\rho_2)$, which we refer
to as the {\em Hecke module of intertwining operators}, of
$H$-maps between the corresponding representation spaces, is a
bi-module with a left $\mathrm{End}_H(\rho_1)$-action and right
$\mathrm{End}_H(\rho_2)$-action, where $\mathrm{End}_H(\rho)=\mathrm{Hom}_H(\rho,\rho)$. It admits a geometric description
as the space of $Q_2$-$Q_1$ bi-invariant functions $\C[Q_2
\backslash H /Q_1]$ and an algebraic description given by
$\oplus_{\sigma \in \mathrm{Irr}(H)}
\Mat(n_{1,\sigma},n_{2,\sigma};\C)$, where $n_{i,\sigma}$ is the
multiplicity of an irreducible representation $\sigma$ of $H$ in
$\rho_i$. In particular,
\begin{equation}\label{dim.equality.intro}
\dim \C[Q_2 \backslash H /Q_1] = \sum_\sigma
n_{1,\sigma}n_{2,\sigma}.
\end{equation}

Let $S_1$, $S_2$ be closed sets of roots of a Chevalley group $G$
(see Section \ref{section-intertwining}) and let $P_{S_1}$,
$P_{S_2}$ be the corresponding parabolic subgroups of $G$. Let
$H=G(\lri,m)$ and $Q_i=P_{S_i}(\lri,m)$ ($i=1,2$), where the
latter are defined as the images in $G(\lri, m)$ of the
$\lri$-points of $P_{S_i}$. For each $m$, put $b^{S_1,S_2}_{\lri, m}=\dim
\C[P_{S_2}(\lri,m) \backslash G(\lri,m) /P_{S_1}(\lri,m)]$. We
define the zeta function of the Hecke modules of intertwiners of
type $(S_1,S_2)$ for $G(\lri)$ to be
$$\zeta^{S_1, S_2}_{G(\lri)}(s)=\sum_{m=0}^{\infty} b^{S_1,S_2}_{\lri, m} q^{-ms}.$$
\begin{thmABC}\label{theorem-hecke}
Let $G$ be a Chevalley group and let $P_{S_1}$, $P_{S_2}$ be
standard parabolics of $G$. There exist a constant $N$, formulae
$\psi_1(x),\dots,\psi_{k}(x)$ in the language of rings and
rational functions $R_1(X,Y)$,...,$R_{k}(X,Y)$ over $\Z$ such that
for all complete discrete valuation rings $\lri$ with finite
residue field $\resfield$ of characteristic $p$ and cardinality
$q$, where $p>N$, and for all $s\in \C$ with $\text{Re}(s)>
\sigma_0$ for some $\sigma_0\in \R$, the zeta function of Hecke
modules of type $(S_1,S_2)$ depends only on $q$ and can be written
as
$$\zeta^{S_1,S_2}_{G(\lri)}(s)=m_1(\resfield)R_1(q,q^{-s})+\dots+m_k(\resfield)R_k(q,q^{-s}),$$
where $m_i(\resfield)$ denotes the cardinality of the set of
$\resfield$-points of $\psi_i(x)$, for $1\leq i \leq k$. In
particular, if $\lri$ and $\lri'$ are compact discrete valuation
rings with residue fields of characteristic larger than $N$ and of
the same cardinality, then $G(\lri)$ and $G(\lri')$ have the same
zeta functions of Hecke modules. Consequently, the numbers
$b^{S_1,S_2}_{\lri, m}$ depend only on~$q$.
\end{thmABC}

To prove Theorems~\ref{theorem-cc} and \ref{theorem-hecke}, we
express the zeta function of each  one of the above types in terms
of integrals over the group $G(\lri) \times G(\lri)$ with respect
to a normalized Haar measure, with integrands defined by certain
definable conditions. This is done by applying an appropriate
counting principle to the finite groups $G(\lri/\maxideal^m)$.
Next, we re-write each integral in terms of integrals over the
affine space $K^d$ with respect to the standard additive Haar
measure, using Theorem~\ref{theorem-coord}. Finally, we use the
fact that the latter are definable with respect to the language
$\cL_{\vf}$. We then apply Theorem B.

\subsection{Notation} $\lfi$ stands for a non-Archimedean local field, $\lri$ its discrete valuation ring, $\resfield$ the residue field,
$\maxideal$ the maximal ideal and $\pi$ a uniformizer.
The valuation on $K$ is denoted $v(\cdot)$. If $G$ is an algebraic
group then its $\lri$-points are denoted by $G(\lri)$, $G^m(\lri)$
is the $m^{\mathrm{th}}$ congruence kernel and $G(\lri,m)$ is the
$m^{\mathrm{th}}$ congruence quotient. For a locally compact
topological group $\mu_G$ stands for a Haar measure on $G$. We use $d_G$ as shorthand for $\dim G$.

\subsection{Organization of the paper}

In Section \ref{MT1} we describe the Denef-Pas language for
Henselian valued fields. Section \ref{MT2} concerns
cell-decomposition and quantifier elimination for local fields. In Sections \ref{MT3} and \ref{measure-transfer} we prove Theorems
 \ref{theorem-model-theory} and \ref{theorem-coord} respectively. In Sections~\ref{section-counting-cc} and \ref{section-intertwining}
we prove Theorems~\ref{theorem-cc} and \ref{theorem-hecke}, respectively, by applying Theorems \ref{theorem-coord}
and \ref{theorem-model-theory}. In the last section we make some concluding remarks.

\subsection{Acknowledgments}

We thank Nir Avni, Raf Cluckers, Anton Evseev, Julia Gordon, Assaf Hasson, Ehud Hrushovski, Tapani Hyttinen,
Jonathan Kirby, Jochen Koenigsmann, Francois Loeser,
Angus Macintyre, Jonathan Pila, Philip Scowcroft, Ilya Tyomkin, Amnon Yekutieli and Boris Zilber for various illuminating discussions during
the course of this work. We are also grateful to the anonymous referee for valuable comments and suggestions. The first author was supported by a postdoctoral fellowship funded by the Skirball Foundation via the
Center for Advanced Studies in Mathematics at Ben-Gurion University of the Negev (with thanks to Nadia Gurevich); his further research visits were supported
by the Academy of Finland and the University Research Committee of the University of Cape Town. The first, second, and fourth authors wish to thank Ben-Gurion University for support and hospitality.

\section{The Denef-Pas language for Henselian valued fields}\label{MT1}

\subsection{}

Let $K$ denote a valued field with valuation $v:K\rightarrow
\Gamma\cup\{\infty\}$, where $\Gamma$ is an ordered abelian group.
Let $\lri$ denote the valuation ring of $K$ and $\resfield$ the
residue field. Let $\mathrm{Res}:\lri\rightarrow \resfield$ denote
the residue map. Let $\maxideal$ denote the maximal ideal of
$\lri$. An angular component map on $K$ modulo $\maxideal$ is a
map $\ac:\lfi \rightarrow\resfield$ satisfying
$\overline{\mathrm{ac}}(0)=0$, $\ac(xy)=\ac(x)\ac(y)$ if $x,y\neq
0$, and $\ac(x)=\mathrm{Res}(x)$ if $v(x)=0$. A cross section on
$K$ is a homomorphism $\cs:\Gamma\rightarrow K^\times$ satisfying
$v\circ \cs=id$, extended to $\Gamma\cup \{\infty\}$ by
$\cs(\infty)=0$.

For any field $F$, the field of Laurent series $F((t))$ carries a
natural angular component map modulo $(t)$ defined by
$\ac(\sum_{i\geq l} a_it^i)=a_l$ where $a_i\in F$ and $a_l\neq 0$,
and a natural cross section defined by $\cs(n)=t^n$ (note that
$\Gamma=\Z$ in this example). If $K$ is a finite extension of
$\Q_p$ with uniformizing element $\pi$ (i.e.\ an element of least
positive value taken to be $1$), then $K$ has an angular component
map modulo $(\pi)$ defined by $\ac(x)=\mathrm{Res}(x\pi^{-v(x)})$ for $x\neq 0$,
and cross section defined by $\cs(n)=\pi^n$. Note that up to a
normalization of the valuation, we may assume in both cases that
the value group is isomorphic to $\Z$ and take $1$ as least
positive element. If a valued field has a cross section $\cs$, then
it also has an angular component map modulo $\maxideal$ defined by
$\ac(x)=\mathrm{Res}(x\cs(-v(x)))$ if $x\neq 0$. In general, by a theorem of
Cherlin \cite{Cherlin}, any $\aleph_1$-saturated valued field has a
cross section, hence an angular component map. Thus any valued
field has an elementary extension with a cross section. It is not
difficult to construct examples of valued fields which do not have
an angular component map modulo $\maxideal$, but these examples
will not concern us.

We shall be considering ultraproducts $\prod_{i\in I}K_i/\frak{U}$
of local fields $K_i$ with respect to an ultrafilter $\frak{U}$
on a countable index set $I$. Such fields have an explicit cross
section and angular component map. Indeed, any
ultraproduct of fields $K_i$ with cross sections~$\cs_i$ has a
cross section $\cs$ defined by $\cs((x_i)^*)=(\cs_i(x_i))^*$ where
$*$ denotes an equivalence class in the ultraproduct.

\subsection{}

The Denef-Pas language $\cL_{\vf}$ for valued fields is a family of
first-order three-sorted
languages of the form
$$(\cL_{\mathrm {Val}},\cL_{\mathrm {Res}},\cL_{\mathrm{Ord}},v,\mathrm{\overline{ac}}),$$ with

(i) the language of rings $\cL_{\mathrm{Val}}=\{+,-,\cdot~,0,1\}$ for the
valued field sort,

(ii) the language of rings $\cL_{\rm{Res}}=\{+,-,\cdot~,0,1\}$ for the residue field
sort,

(iii) the language $\cL_{\mathrm{Ord}}$, which is an extension of the language
$\{+,\infty,\le\}$ of ordered abelian groups with an element
$\infty$, which is interpreted as a top element, for the value
group sort,

(iv) a function symbol $v$ from the valued field sort to the value
group sort, which is interpreted as a valuation,

(v) a function symbol $\ac$ from the valued field sort to the
residue field sort, which is interpreted as an angular component
map modulo $\maxideal$.

Any fixed extension of the language $\{+,\infty,\leq\}$ gives an example of
the Denef-Pas language.
\subsection{}

We shall work in an example of $\cL_{\vf}$ which is the language
$$\cL_{\vfp}:=(\cL_{\mathrm {Val}},\cL_{\mathrm {Res}},\cL_{\mathrm{Pres}\infty},v,\mathrm{\overline{ac}})$$
where $\cL_{\mathrm{Pres}\infty}=\cL_{\mathrm{Pres}}\cup
\{\infty\}$ and $\cL_{\mathrm{Pres}}=\{+,0,1,\le\}\cup\{\equiv_n:
n\geq 1\}$ is the Presburger language for ordered abelian groups.
If we interpret $\equiv_n$ as congruence modulo $n$, then a
non-Archimedean local field or an ultraproduct of such fields will
be a structure for $\cL_{\vfp}$, where the value groups are
respectively $\Z$ and an ultrapower of $\Z$.

\subsection{}

The letters $x,y,u,\dots$ will denote tuples of valued field
variables, $\xi,\rho,\eta,\dots$ tuples of residue field
variables, and $k,l,\dots$ tuples of value group variables.
Formulae of $\cL_{\vf}$ are built up with the rules of many-sorted
language using the logical connectives $\wedge$ (and), $\vee$
(or), $\neg$ (negation), the quantifiers $\exists$ (exists) and
$\forall$ (for all), and $=$(equality). Formulae will be written
in the form $\psi(x,\xi,k)$, where $x,\xi,k$ are valued field,
residue field, and value group variables, respectively. For more
information on many-sorted logic see \cite{Enderton}, pp.\
277-281.

We shall be considering valued fields
$(K,\resfield,\Gamma\cup\{\infty\},v,\ac)$ as a structure for
$\cL_{\vf}$ where the function symbols $v$ and $\ac$ are
interpreted as valuation and angular component modulo $\maxideal$
respectively. By abuse of language we abbreviate this as $K$.
Given such a field $K$ and an $\cL_{\vf}$-formula
$\psi(x,\xi,k)$, where $x$ is an $m$-tuple, $\xi$ an
$n$-tuple, and $k$ an $s$-tuple, the definable set in $K$ defined
by $\psi(x,\xi,k)$ is the set of all $(a,d,c)\in K^m\times
\resfield^n\times \Gamma^s$ such that  $K\models \psi(a,d,c)$.

Throughout, any \lq definable\rq~ object stands for
$\cL_{\vf}$-definable unless specified otherwise.

\section{Quantifier elimination and cell decomposition}\label{MT2}
In this section we obtain uniform cell decomposition and
quantifier elimination for non-Archimedean local fields of large
residue characteristic using the Denef-Pas cell decomposition
theorem. Other closely related cell decomposition theorems have
been obtained in \cite{CL1,CL2,DL,Macintyre,Pas}.

Let $T$ be a first-order theory in a multi-sorted
language $\cL$. Let $\sort$ be a sort in $\cL$. Then $T$ is said to
have quantifier elimination in $\cL$ in the sort $\sort$ if for every $\cL$-formula
$\phi(x)$ there is an $\cL$-formula $\psi(x)$ which has no quantifiers of the sort
$\sort$
such that $T\vdash \forall x (\phi(x)\leftrightarrow \psi(x))$. If $M$ is an
$\cL$-structure, then $M$ is said to have quantifier elimination in $\cL$ in the sort $\sort$
if the $\cL$-theory of $M$ has this property.

Let $\cT_{\hac,0}$ denote the $\cL_{\vf}$-theory of Henselian
valued fields $K$ of residue characteristic zero such that there
is an angular component map $\ac:K\to \resfield$. Note that the
condition of being Henselian can be expressed by countably many sentences.

\begin{Def}[{\bf Cells}, cf. \protect{\cite[Definition 2.9]{Pas}}] Let $t$ and $x=(x_1,\dots,x_m)$ be valued field sort variables and $\xi=(\xi_1,\dots,\xi_n)$
residue field sort variables. Let $C$ be a definable subset of
$K^m\times \resfield^n$. Let $b_1(x,\xi), b_2(x,\xi),c(x,\xi)$ be $\cL_{\vf}$-definable functions 
from $C$ to $K$, $\lambda$ a positive integer,
and $\Diamond_1,\Diamond_2$ be $<,\leq$ or no condition. For each
$\xi\in \resfield^n$, let $A(\xi)$ denote the set of all $(x,t)$
from $K^m\times K$ such that  $(x,\xi)\in C,$ $\ac(t-c(x,\xi))=\xi_1$ and
$v(b_1(x,\xi))\Diamond_1 \lambda v(t-c(x,\xi)) \Diamond_2
v(b_2(x,\xi))$.

Suppose that the $A(\xi)$ are distinct; then $A=\bigcup_{\xi\in
\resfield^n} A(\xi)$ is called a cell in $K^m\times K$ with
parameters $\xi$ and centre $c(x,\xi)$, and $A(\xi)$ is called a
fiber of the cell. Suppose $C$ is defined by a formula $\psi(x,\xi)$ of $\cL_{\vf}$. 
Following \cite{Macintyre}, we call $\Theta=(\psi(x,\xi),b_1(x,\xi),b_2(x,\xi),c(x,\xi),\lambda)$ the datum of 
the cell $A$ and we say that $A$ is defined by $\Theta$.
\end{Def}

\begin{Def}[{\bf Uniform cell decomposition}, cf. \protect{\cite[Theorem 3.2]{Pas}}]\label{def-cd}
Let $t$ and $x=(x_1,\dots,x_m)$ be valued field sort variables and $\xi=(\xi_1,\dots,\xi_n)$
residue field sort variables. Let $\mathcal{F}$ be a class of valued fields. Let $f_1(x,t),\dots,f_r(x,t)$ be polynomials in $t$ whose
coefficients are definable functions in $x$ without parameters from $K^m$ to $K$ for all $K\in \mathcal{F}$, i.e.\ expressions of
the form $\sum_{k} a_k(x)t^k$ where the $a_k(x)$ are definable 
functions in $x$ without parameters from $K^m$ to $K$. For example, polynomials in $x,t$ with integer coefficients have this
form. We say that the fields in $\mathcal{F}$ have uniform cell decomposition of
dimension $m$ with respect to the polynomials $f_1(x,t), \ldots,
f_r(x,t)$ if, for some positive integer $N$, there exist
 
\begin{itemize}
\item  a non-negative integer $n$ and $\cL_{\vf}$-formulas $\psi_i(x,\xi)$, independent of $K\in \mathcal{F}$, defining sets without parameters in 
$K^m\times \resfield^n$;
\item $\cL_{\vf}$-formulas $b_{i,1}(x,\xi),b_{i,2}(x,\xi),c_i(x,\xi)$ and $h_{ij}(x,\xi)$, independent of ${K\in \mathcal{F}}$, defining functions without parameters from $K^m\times \resfield^n$ to $K$;
\item positive integers $\lambda_i$;
\item non-negative integers $w_{ij}$;
\item functions $\mu_{i}:\{1,\dots,r\}\to \{1,\dots,n\}$
\end{itemize}
where $i =1,\dots,N$ and $j=1,\dots,r$, such that

\begin{itemize}
\item for all $ i\in \{1,\dots,N\}$, $\Theta_i=(\psi_i(x,\xi),b_{i,1}(x,\xi),b_{i,2}(x,\xi),c_i(x,\xi),\lambda_i)$ is 
a cell datum, and for every field $K$ in $\mathcal{F}$ it defines a cell in $K^m\times K$;
\item for each field $K$ from $\mathcal{F}$, the set $K^m\times K$ is the union of the cells $A_i$ defined by $\Theta_i$ 
in $K$ for $i=1,\dots,N$;
\item for all $i\in \{1,\dots,N\}$, $j\in \{1,\dots,r\}$ and $(x,t)\in A_i(\xi)$, we have
\begin{displaymath}
\begin{array}{rcl}
v( f_j(x,t))&=&v\left(h_{ij}(x,\xi)(t-c(x,\xi))^{w_{ij}}\right);\\
\ac(f_j(x,t))&=&\xi_{\mu_i(j)}
\end{array}
\end{displaymath}
where the integers $w_{ij}$ and the maps $\mu_i$ are independent of $(x,\xi,t)$ and independent of the field $K$.
\end{itemize}
\end{Def}

In the following, by slight abuse of language, we shall sometimes denote a datum 
by $(C,b_1(x,\xi),b_2(x,\xi),c(x,\xi),\lambda)$
by which we shall mean that $C$ is the set of realizations in $K\in \mathcal{F}$ of a 
formula which is fixed for all fields in $\mathcal{F}$.

The following cell decomposition theorem was proved by Pas
\cite{Pas} based on the cell decomposition theorem of Denef for a
finite extension of $\Q_p$, see \cite{Denef}, which was in turn
based on a cell decomposition theorem of P.J.Cohen \cite{Cohen}.

\begin{theo}\label{Pas-cd}\cite[Theorem 3.2]{Pas}
Let $t$ and $x=(x_1,\dots,x_m)$ be valued field sort variables.
Let $f_1(x,t),\dots,f_r(x,t)$ be polynomials in $t$ with
coefficients definable functions in $x$. Then the class of all
models of $\cT_{\hac,0}$ has uniform cell decomposition of
dimension $m$ with respect to the polynomials $f_i(x,t)$ for
$i=1,\cdots,r$.
\end{theo}

The following quantifier elimination theorem follows, with (2) an easy consequence of (1).

\begin{theo} \label{Pas-qe}\cite[Theorem 4.1]{Pas}
\begin{enumerate}
 \item The theory $\cT_{\hac,0}$ admits elimination of quantifiers in the
valued field sort.
\item Every $\cL_{\vf}$-formula $\psi(x,\xi,k)$ without parameters, where $x$
is a tuple of variables in the valued field sort, $\xi$ a tuple
of variables in the residue field sort, and $k$ a tuple of
variables in the value group sort, is $\cT_{\hac,0}$-equivalent to
a finite disjunction of formulae of the form
$$\chi(\ac(g_1(x)),\dots,\ac(g_s(x)),\xi)\wedge \theta(v(g_1(x)),\dots,v(g_s(x)),k),$$
where $\chi(x,\xi)$ is an $\cL_{\mathrm{Res}}$-formula,
$\theta(x,k)$ an $\cL_{\mathrm{Ord}}$-formula, and\\
$g_1(x),\dots,g_r(x)\in~\Z[x]$. 
\end{enumerate}

\end{theo}

Now we consider how to transfer statements from models of
$\cT_{\hac,0}$ to non-Archimedean local fields.

\begin{lemma}\label{compactness} Let $\phi$ be a sentence of $\cL_{\vf}$. Assume that $\phi$ holds in all models of
$\cT_{\hac,0}$. Then there exists some constant $c=c(\phi)$ such
that $\phi$ holds in all non-Archimedean local fields $K$ with
residue characteristic at least $c$.\end{lemma}

\begin{proof} Suppose the contrary; then there are infinitely many non-Archimedean local fields $K_1,\dots,K_i,\dots$
of increasing residue characteristic such that for all $n\geq 1$,
$K_n\models \neg \phi$. Let $\frak{U}$ be a non-principal
ultrafilter on the set of natural numbers $\N$, and $L$ the
ultraproduct $\prod_{n\in \N} K_n/\frak{U}$. Then $L\models \neg
\phi$. But $L$ has residue characteristic zero and is an
$\cL_{\vf}$-structure, hence is a model of $\cT_{\hac,0}$, which
is a contradiction.
\end{proof}

\begin{theo}\label{cd} Let $f_1(x,t),\dots,f_r(x,t)$ be as in Theorem \ref{Pas-cd}.
Then there is some constant $c=c(m; f_1(x,t),\dots,f_r(x,t))$
such that the class of all non-Archimedean local fields of residue
characteristic at least $c$ has uniform cell decomposition of
dimension $m$ with respect to the polynomials $f_i(x,t)$ for
$i=1,\dots,r$.
\end{theo}

\begin{proof} By Theorem \ref{Pas-cd}, the class of all models of $\cT_{\hac,0}$ has uniform
cell decomposition with respect to the polynomials
$f_1,\dots,f_r$. For $i\in \{1,\dots,N\}$, denote the cell data by
$\Theta_i(C_i,b_{i,1}(x,\xi),b_{i,2}(x,\xi),c_i(x,\xi),\lambda_i)$.
Let $h_{ij}(x,\xi)$, $v_{ij}$, and $\mu_i$ be as in
Definition~\ref{def-cd}. Then every model of $\cT_{\hac,0}$
satisfies the conjunction $\psi_1\wedge \psi_2$, where $\psi_1$ is
the formula
$$\forall x_1\dots \forall x_m\forall t \exists \xi \left(\bigvee_{1\leq i\leq N} (x,\xi)\in C_i \wedge
v(b_{i,1}(x,\xi))\Diamond_1 \lambda_i v(t-c(x,\xi)) \Diamond_2
v(b_{i,2}(x,\xi))\right)$$ and $\psi_2$ is the formula
$$\left.\bigwedge_{1\leq i\leq N} \right( \forall x_1 \dots \forall x_m \forall t \forall \xi [(x,\xi)\in C_i \wedge
v(b_{i,1}(x,\xi))\Diamond_1 \lambda_i v(t-c(x,\xi)) \Diamond_2
v(b_{i,2}(x,\xi))]$$
$$\left.\longrightarrow \bigwedge_{1\leq j\leq r} \left[v(f_j(x,t))=v(h_{ij}(x,\xi)(t-c(x,\xi))^{v_{ij}}) \wedge
\ac(f_j(x,t))=\xi_{\mu_i(j)}\right]\right).$$

By Lemma \ref{compactness} there is some $c$ such that the
sentence $\psi_1\wedge \psi_2$ holds in all non-Archimedean local
fields $K$ of residue characteristic at least $c$. For all such
$K$, the space $K^m$ has cell decomposition as specified by $\psi_1\wedge
\psi_2$, hence the cell decomposition is uniform.
\end{proof}

\begin{theo}\label{qe} Let $\psi(x,\xi,k)$ be an $\cL_{\vf}$-formula. Then there is a constant $c$ and
an $\cL_{\vf}$-formula $\phi(x,\xi,k)$ without quantifiers of the
valued field sort such that $\phi(x,\xi,k)$ and $\psi(x,\xi,k)$
are equivalent in all non-Archimedean local fields of residue
characteristic at least~$c$. Furthermore, $\phi(x,\xi,k)$ can be
expressed as a finite disjunction of formulae of the form
$$\chi(\ac(g_1(x)),\dots,\ac(g_s(x)),\xi)\wedge \theta(v(g_1(x)),\dots,v(g_s(x)),k),$$
where $\chi(x,\xi)$ is an $\cL_{\mathrm{Res}}$-formula,
$\theta(x,k)$ an $\cL_{\mathrm{Ord}}$-formula, and
$g_1(x),\dots,g_r(x)\in \Z[x]$.
\end{theo}

\begin{proof} Follows from Theorem \ref{Pas-qe} and Lemma \ref{compactness}.
\end{proof}

\section{Integration and definable sets}\label{MT3}

In this section we prove Theorem \ref{theorem-model-theory} using the methods of Pas
\cite{Pas}.

\subsection{Integrals as sums over the residue field and value group}
 The following
theorem shows that the measure of a subset of $K^{m+1}$ defined by
a formula in the three-sorted language $\cL_{\vf}$ can be
expressed in terms of measures of subsets which are defined by a
quantifier-free formula in the value field sort and involve one
less valued field variable. This comes at the cost of introducing
extra residue field and value group variables. The theorem
generalizes Lemma 5.2 in \cite{Pas}. With Theorems~\ref{cd} and
\ref{qe} at hand, Pas' proof easily adapts to our case.

\begin{theo}\label{lem-int} Let $y,\rho,k$ be (tuples of) valued field, residue field, and value group variables respectively. Let
$x=(x_1,\dots,x_m)$ and $t$ be valued field variables and
$\psi(x,t,y,\rho,k)$ an $\cL_{\vf}$-formula without quantifiers of
valued field sort. Consider the parametrized integral
$$\int_{W_K(y,\rho,k)} dxdt$$
where $W_K(y,\rho,k)=\{(x,t)\in K^{m+1}:
K\models \psi(x,t,y,\rho,k)\}$, and $dx$ (respectively $dt$) is the
normalized Haar measure on $K^m$ (respectively $K$) such that $\lri^m$
(respectively $\lri$) has measure $1$. Then there are
$\cL_{\vf}$-formulae $\phi_1(x, \xi^{(1)}, l, y, \rho, k),\ldots,
\phi_N(x, \xi^{(N)}, l, y, \rho, k)$ without quantifiers of valued
field sort, where $\xi^{(i)}$ and $l$ are residue field sort and
value group sort variables, of arities $n(i)$ and $1$,
respectively, and a constant $c$ such that, for all
non-Archimedean local fields $K$ of residue characteristic at
least $c$,

$$\int_{W_K(y,\rho,k)} dxdt=q^{-1}\sum_{1\leq i\leq N} \sum_{\xi^{(i)}\in \resfield^{n(i)}}\sum_{l\in \Z} q^{-l}
\int_{E_K^{(i)}(\xi^{(i)},l,y,\rho,k)} dx$$ where
$$E_K^{(i)}(\xi^{(i)},l,y,\rho,k)=\{x\in K^m: K\models
\phi_i(x,\xi^{(i)},l,y,\rho,k)\}.$$
\end{theo}

\begin{proof} In the formula $\psi$, the variable $t$ appears only in the terms of the form $v(f(y,x,t))$ and
$\ac(f(y,x,t))$, where $f(y,x,t)$ is a polynomial in $(y,x,t)$ with
integer coefficients. With respect to all such polynomials occurring in $\psi$, we apply cell
decomposition to the $(y,x,t)$-space $K^{e+m}\times K$, where $y, x$ are $e$- and $m$-tuples respectively. By Theorem
\ref{cd}, there is some constant $c$ such that for any
non-Archimedean local field $K$ of residue characteristic at least
$c$, the space $K^{e+m}\times K$ has uniform cell decomposition
with respect to the polynomial $f(y,x,t)$. Let
$A_K^{(1)},\dots,A_K^{(N)}$ denote the cells. For each~$y$, let
$B_K^{(i)}(y)=\{(x,t)\in K^m\times K: (y,x,t)\in A_K^{(i)}\}$.
Then

$$\int_{W_K(y,\rho,k)} dxdt=\sum_{1\leq i\leq N} \int_{W_K(y,\rho,k)\cap B_K^{(i)}(y)} dxdt.$$

Let $A_K$ denote one of the cells $A_K^{(i)}$, and $B_K(y)$ the
corresponding $B_K^{(i)}(y)$. It suffices to compute the integral
$\int_{W_K(y,\rho,k)\cap B_K(y)} dxdt$.

The cell $A_K$ has some parameters $\xi=(\xi_1,\dots,\xi_n)$ and
is specified by the datum $(C, b_1(x, \xi), b_2(x, \xi), c(x,
\xi),\lambda)$. By cell decomposition, if $(y,x,t)\in A_K$, then
$(y,x,t)\in A(\xi)$, thus for some $h,c,w,\mu$ as in the
definition of a cell, and depending on $f(y,x,t)$, we have
\begin{displaymath}
\begin{array}{lrcl}
&v(f(y,x,t))&=&v(h(y,x,\xi))+w v(t-c(y,x,\xi)), ~\text{and}\\
&\ac(f(y,x,t))&=&\xi_{\mu(f)}.
\end{array}
\end{displaymath}

Thus on $A_K$ the formula $\psi(x,t,y,\rho,k)$ is equivalent to a
formula, without valued field sort quantifiers, of the form
$\theta(x,\xi,v(t-c(y,x,\xi)),y,\rho,k)$. The set
\nobreak{$W_K(y,\rho,k)\cap B_K(y)$} can be written in the form
$$\bigcup_{\xi\in \resfield^n}\{(x,t)\in K^m\times K:
v(b_1(y,x,\xi)) \Diamond_1 \lambda v(t-c(y,x,\xi))\Diamond_2
v(b_2(y,x,\xi)),$$
$$\ac(t-c(y,x,\xi))=\xi_1, (y,x,\xi)\in C, K\models \theta(x,\xi,v(t-c(y,x,\xi))),y,\rho,k)\}.$$
Since the valuation takes discrete values, this set can be written
as
$$\bigcup_{\xi\in \resfield^n}\bigcup_{l\in \Z}\{(x,t)\in K^m\times K: x\in E_K(\xi,l,y,\rho,k),
v(t-c(y,x,\xi))=l,\ac(t-c(y,x,\xi))=\xi_1\},$$ where
$E_K(\xi,l,y,\rho,k)$ denotes the set $$\{x\in K^m: (y,x,\xi)\in
C, v(b_1(y,x,\xi))\Diamond_1 \lambda l\Diamond_2 v(b_2(y,x,\xi)),
K\models \theta(x,\xi,l,y,\rho,k)\},$$ which is clearly defined
with no quantifiers of the valued field sort. Therefore

$$\int_{W_K(y,\rho,k)\cap B_K(y)} dxdt=\sum_{\xi\in \resfield^n} \sum_{l\in \Z}
\int_{E_K(\xi,l,y,\rho,k)}\left(\int_{\substack{v(t-c(y,x,\xi))=l \\
\ac(t-c(y,x,\xi))=\xi_1}} dt\right) dx.$$ Doing a linear change of
variables $t-c(y,x,\xi)\mapsto t$, we can assume that
$c(y,x,\xi)=0$. The inner integral thus becomes
$$\int_{\substack{v(t)=l \\ \ac(t)=\xi_1}} dt=q^{-l-1}.$$
The proof is complete.
\end{proof}

We now evaluate sums over definable sets in the value group sort.

\begin{lemma}\label{press2} Let $\psi(l_1,\dots,l_m,n)$ be a formula of $\cL_{\mathrm{Pres}\infty}$. Let
\[
\begin{split}
&E=\left\{(l_1,\dots,l_m,n) \in \Z^{m+1}: \Z\cup\{\infty\} \models
\psi(l_1,\dots,l_m,n)\right\} \quad \text{and}\\
&J(s,q)=\sum_{(l_1,\dots,l_m,n)\in E} q^{-ns-l_1-\dots-l_m}.
\end{split}
\]
Suppose that for some open set $S$ in $\R$ and some set $A$ of
positive integers, the sum $J(s,q)$ converges for $s\in S$ and
$q\in A$. Then there exist polynomials $P,Q\in \Z[X,Y]$ such that
$$J(s,q)=\frac{P(q,q^{-s})}{Q(q,q^{-s})}$$
for all $s\in S$ and for all $q\in A$.
\end{lemma}
\begin{proof}

The proof of \cite[Lemma~5.6]{Pas} goes through in our slightly more general setting in which $A$ is an arbitrary set of positive integers (in \cite{Pas}, $A$ is a cofinite set of primes).

\end{proof}

\subsection{Proof of Theorem \ref{theorem-model-theory}}

We assume that the integral converges for $\text{Re}(s)>\sigma_0$. Take $K$ of sufficiently large residue characteristic, say $c$. Throughout the proof, each time we apply Theorems \ref{qe} and
\ref{lem-int}, we need to increase $c$ to exclude further primes.

Note that
\[
\int_{W_K} |\Phi(x)|^s~ dx=\sum_{n \in \Z} q^{-ns}
\cdot \int_{A_n} dx,
\]
where $A_n=\{ a \in W_K: v(\Phi(a))=n\}$. By Theorem \ref{qe}, we may
assume that $\psi(x)$ and $\Phi(x)$ do not contain valued field quantifiers
and using Theorem \ref{lem-int} repeatedly, we deduce that there exists a
positive integer $e$ such that $\int_{A_n} dx$ can be written uniformly
for all non-Archimedean local fields $K$ of sufficiently large
residue characteristic as a finite sum of expressions of the form
\[
q^{-m} \sum_{\xi\in \resfield^e}\sum_{\substack{l_1,\dots,l_m \in \Z \\
\phi(\xi,l_1,\dots,l_m,n)}}q^{-l_1-\dots-l_m},
\]
where by Theorem \ref{qe}, $\phi$ does not contain valued field variables.
By \cite[Lemma 5.3]{Pas}, $\phi$ is equivalent to
\[
\bigvee_{i=1}^k(\chi_i(\xi) \wedge \theta_i(l_1,\dots,l_m,n))
\]
where  $\chi_i$ is an $\cL_{\mathrm{Res}}$-formula and $\theta_i$ an
$\cL_{\mathrm{Pres}\infty}$-formula for each $i$. Thus $\int_{W_K}
|\Phi(x)|_{\maxideal}^s~ dx$ is a finite sum of expressions of the form
\[
q^{-m} ~\sum_{i=1}^k\sum_{\substack{\xi \in \resfield^e \\ \resfield \models \chi_i(\xi)}} 1 \sum_{\substack{l_1,\dots,l_m,n \in \Z \\
\Z \cup\infty \models
\theta_i(l_1,\dots,l_m,n)}}q^{-ns-l_1-\dots-l_m}.
\]
Note that $\sum_{\substack{\xi\in \resfield^e \\ \chi_i(\xi)}} 1$ is
the cardinality of the set defined by $\chi_i(\xi)$ in
$\resfield^e$.
It remains to show that the sums $\sum_{\substack{l_1,\dots,l_m,n \in \Z \\
\theta_i(l_1,\dots,l_m,n)}}q^{-ns-l_1-\dots-l_m}$ can be written
as rational functions over $\Z$ uniformly in $q$. These sums are
convergent for all $s\in \R_{>\sigma_0}$ since all the
terms in the sum are positive and $\int_{W_K}|\Phi(x)|^s~ dx$ is convergent for $\text{Re}(s)>\sigma_0$. By Lemma \ref{press2}, $\int_{W_K}|\Phi(x)|^s~ dx$ is equal to a rational function $P(q, q^{-s})/Q(q,q^{-s})$ on some open subset $S$ of $\R$. Since the integral is holomorphic, it is equal to this rational function on some suitable half-plane, and the proof is complete.

\begin{cor}\label{cor-of-B} Let $x$ be an $m$-tuple of field variables. Let $W_K$ be as in Theorem \ref{theorem-model-theory}.
Let $\Theta: K^m\to \Z$ be an $\cL_{\vf}$-definable function.
Assume there is some $\sigma_0\in \R$ such that the integral
$\int_{W_K} q^{-\Theta(x)s} dx$ converges for all $s\in \C$ with
$\text{Re}(s)> \sigma_0$ for all non-Archimedean local fields $K$
of sufficiently large residue characteristic. Then there exist a
constant $N$, rational functions $R_1(X,Y),\dots,R_k(X,Y)$ over
$\Z$, and formulae $\psi_1(x),\dots,\psi_k(x)$ in the language of
rings such that for all non-Archimedean local fields $K$ with
residue field $\resfield$ of characteristic at least $N$, and for
all $s\in \C$ with $\text{Re}(s)> \sigma_0$, we have
$$\int_{W_K} q^{-\Theta(x)s} dx=m_1(\resfield)R_1(q,q^{-s})+\dots+m_k(\resfield)R_k(q,q^{-s}).$$
where $m_i(\resfield)$ denotes the cardinality of the set of
$\resfield$-points of $\psi_i(x)$, $q$ is the cardinality of the
residue field $\resfield$, and $dx$ denotes the Haar measure on
$K^m$ normalized such that $\lri^m$ has measure $1$.
\end{cor}

\begin{proof} By \cite[Proposition 4.15]{Avni} there exist a constant $N$, positive integers $n_1,...,n_l$,
formulae $\phi_1(x),...,\phi_l(x)$ and rational functions
$Q_1,...,Q_l \in \mathbb{Q}(x)$, such that for all non-Archimedean
local fields of residue characteristic greater than $N$, the
domain $K^m$ is a union of the sets $A_i$ defined by $\phi_i$ and
for all $x \in A_i$ one has $\Theta(x)=\frac{1}{n_i}v(Q_i(x))$.
Therefore,
$$\int_{W_K} q^{-\Theta(x)s} dx=\sum_{i=1}^l\int_{A_i \cap W_K} |Q_i(x)|^{s/n_i} dx,$$
and by applying Theorem \ref{theorem-model-theory} on each of the
summands the result follows.

\end{proof}

\section{Measure-preserving change of variables}\label{measure-transfer}

\subsection{Chevalley groups}
We recall the construction of Chevalley groups (see
\cite{Steinberg} for more details). Let $\mathfrak{g}$ be a
$d$-dimensional semisimple complex Lie algebra and let
$\mathfrak{h}$ be a Cartan subalgebra. Let $\Sigma \subset
\mathfrak{h}^* $ be the set of roots of $\mathfrak{g}$. Let
$\Sigma^+=\{\alpha_1,\ldots,\alpha_r\}$ be a set of positive
roots, $\Sigma^\circ=\{\alpha_1,...,\alpha_\ell\} \subset
\Sigma^+$ the simple roots and set $\Sigma^-=\Sigma \backslash
%\smallsetminus
\Sigma^+$. Then $\mathfrak{g}$ decomposes as $\mathfrak{h} \oplus
(\oplus_{\alpha \in \Sigma} \mathfrak{g}_\alpha)$ with respect to
the adjoint $\mathfrak{h}$-action.

Let $\{X_\alpha \suchthat \alpha \in \Sigma\} \cup \{H_\alpha
\suchthat \alpha \in \Sigma^\circ\}$ be a Chevalley basis for
$\mathfrak{g}$. For $\alpha \in \Sigma$ the map $x_\alpha:
\mathbb{G}_a \to \mathrm{GL}_d$  defined by $x_\alpha(t):=\exp(t
~\!\! \mathrm{ad}(X_\alpha))$ is a polynomial isomorphism defined
over $\mathbb{Q}$ onto its image $\frak X_\alpha$. For any simple
root $\alpha \in \Sigma^\circ$ let
$h_{\alpha}(t):=x_{\alpha}(t)x_{-\alpha}(-t^{-1})x_{\alpha}(t)$,
which is an isomorphism $\mathbb{G}_m \to \mathcal H_\alpha
\subset \mathrm{GL}_d$. Let
\[
U= \frak X_{\alpha_1}\dots \frak X_{\alpha_{r}}, \qquad T=
\mathcal H_{\alpha_1}\dots \mathcal H_{\alpha_{\ell}}, \qquad
U^-=\frak X_{-\alpha_1}\dots \frak X_{-\alpha_{r}},
\]
where we have fixed a total ordering on the roots which refines
the standard partial ordering to make the product uniquely
defined. The group $G$ generated by all the groups $\frak
X_\alpha$ ($\alpha \in \Sigma$) is called a Chevalley group. For
any local field $K$ the group of $K$-points of $G$ is locally
compact and hence admits a Haar measure, which is both left and
right invariant. Our aim is to show that it is expressible as a definable integral.
We have the following maps
\[
\mathbb{G}_a^{\Sigma^-} \times \mathbb{G}_m^{\Sigma^\circ} \times
\mathbb{G}_a^{\Sigma^+} \overset{\Psi}{\longrightarrow} U^- \times
T \times U \overset{\Phi}{\longrightarrow} G,
\]
where $\Psi=(\Psi^-,\Psi^\circ, \Psi^+)$ is defined by
\[
\Psi^\pm((a_\alpha)_{\alpha \in \Sigma^\pm})=\prod_{\alpha \in
\Sigma^\pm} x_\alpha(a_\alpha) \quad \mathrm{and} \quad
\Psi^\circ((a_\alpha)_{\alpha \in \Sigma^\circ})=\prod_{\alpha \in
\Sigma^\circ} h_\alpha(a_\alpha),
\]
and $\Phi$ is the product map: $\Phi(x,t,y)=xty$. It is well known
(cf.\ \cite[p.~61, Theorem~7]{Steinberg}) that $\Phi$ is an isomorphism of algebraic
varieties whose image is $U^-TU$. For any
local field $K$, the measure of $U^-(K)T(K)U(K)$ in $G(K)$ is full.

\subsection{Proof of Theorem \ref{theorem-coord}}

For a locally compact group $G$ we let $\Delta_G:G \to \R^+$
denote its modular function. It can be viewed also as the modulus
of the automorphism defined by conjugation. The following theorem
is stated in \cite[Theorem 8.32]{Knapp} for real Lie groups, however, the proof works
{\em mutatis mutandis} for locally compact groups.

\begin{theo}\label{de-haar} Let $G$ be a locally compact group and let $S$ and $T$ be closed subgroups such that
$S\cap T$ is compact, the multiplication map $S\times T\to G$ is an open map,
and the set of products $ST$ exhausts $G$ except possibly for a
set of Haar measure zero. Let $\Delta_T$ and $\Delta_G$ denote the
modular functions of $T$ and $G$ and let $ds$ and $dt$ denote left Haar measures on $S$ and $T$ respectively. Then the measure $\nu$ on $G$
defined by
$$\int_{G} f d\nu=\int_{S\times T} f(st) \Delta_T(t)/\Delta_G(t) ds dt$$
for all measurable functions $f\geq 0$ on $G$, is a left Haar
measure. \end{theo}

The following is a more elaborate version of Theorem
\ref{theorem-coord}.

\begin{theo}\label{theorem-coord2} Let $G$ be a Chevalley group defined over $\mathbb{Z}$ of dimension $d$.
Let $\Phi$ and $\Psi$ be as above. Let $K$ be a non-Archimedean
local field with ring of integers $\lri$ and residue field
$\resfield$ of cardinality $q$. Let $d\lambda$ denote the additive
Haar measure on $K^d=K^{\Sigma^-} \times K^{\Sigma^\circ} \times
K^{\Sigma^+}$ normalized on $\lri^d$. Let $\mu$ be the measure
given by
\begin{equation}\label{haar.integral.formula}
\int_{G(K)} f(g) d\mu=k_G(q)^{-1}\int_{K^{\Sigma^-} \times
K^{\Sigma^\circ} \times K^{\Sigma^+}} \prod_{\beta \in
\Sigma^\circ} |a_\beta|^{-\langle \rho,\beta\rangle-1}
f(\Phi(\Psi(a)))~ d\lambda,
\end{equation}
as $f$ runs through all complex valued Borel functions on $G(K)$,
where $\rho=\sum_{\delta \in \Sigma^+} \delta$,
$a=\left((a_\alpha)_{\alpha \in \Sigma^-},(a_\beta)_{\beta \in
\Sigma^\circ}, (a_\gamma)_{\gamma \in \Sigma^+}\right) \in K^d$
and $k_G(q)=|G(\resfield)|/q^{\dim G}$. Then $\mu$ is a left and
right Haar measure on $G(K)$ normalized such that
$\mu(G(\lri))=1$.
\end{theo}

\begin{proof}

Throughout we drop the $K$ when writing $U(K), U^-(K), \frak X_{\alpha}(K)$. Let $da$ denote an additive Haar measure on $K$ and let
$d\mu_{\alpha}={x_{\alpha}}_{*}(da)$ be the push forward of $da$
along $x_\alpha$, defined by
$$\int_{\frak X_{\alpha}} f(g)d\mu_{\alpha}=\int_K f(x_{\alpha}(a)) da$$
for all Borel functions $f$. Then the measure $\mu_{\alpha}$ is a
Haar measure on $\frak X_{\alpha}$. Similarly, by pushing forward
the Haar measure $da/|a|$ on $K^\times$ to ${\mathcal H}_{\alpha}$
we obtain a Haar measure on ${\mathcal H}_{\alpha}$ as
$d\tau_{\alpha}={h_{\alpha}}_{*}(da/|a|)$. On $U$ (and similarly on $U^-$) we define a measure via the
formula
\begin{equation}\label{Haar.on.U}
\int_{U} f(u) du:=\int_{\prod_{\alpha \in \Sigma^+} \frak
X_\alpha} f\left(\prod_{\alpha \in
\Sigma^+}x_\alpha(a_\alpha)\right) \prod_{\alpha \in
\Sigma^+}da_{\alpha}.
\end{equation}
The measure $du$ on $U$, and similarly the measure $du^-$ on
$U^-$, is indeed a Haar measure by repeated use of Theorem
\ref{de-haar}, utilizing the fact that any product of the $\frak
X_\alpha$'s is unipotent, and the unimodular function of a unipotent group is trivial (being equal to the determinant of unipotent matrices defined by conjugation with unipotent elements).
For $T$ we have
\begin{equation}\label{Haar.on.T}
\int_{T} f(t) dt:=\int_{\prod_{\beta \in \Sigma^\circ} \mathcal
H_\beta} f\left(\prod_{\beta \in
\Sigma^\circ}h_\beta(a_\beta)\right) \prod_{\beta \in
\Sigma^\circ}\frac{da_{\beta}}{|a_\beta|},
\end{equation}
which is clearly a Haar measure (the modular functions in Theorem~\ref{de-haar} are trivial since T is abelian).

Since the complement of the $K$-points of $U^-TU$ in $G(K)$ has
measure zero and the groups $U$ and $U^-$ are unimodular, the
measure defined by
\[
\int_{U^-\times T\times U} \frac{1}{\Delta_{U^-T}(t)}
f(\Phi(u^-,t,u))du^- dt du
\]
is a Haar measure on $G(K)$ by Theorem \ref{de-haar}. Since the
value of the modular function at $t$ is the modulus of conjugation
$u \mapsto \mathrm{mod}(t u t^{-1})$, we can compute
$\Delta_{U^-T}(t)$ using the following equality \cite[Lemma
20(c)]{Steinberg}
\begin{equation}\label{torus.action.single}
h_\beta(b)x_{\alpha}(a)h_\beta(b)^{-1}=x_\alpha(b^{\langle
\alpha,\beta\rangle}a), \qquad \alpha,\beta \in \Sigma,~0 \ne
b,a\in K,
\end{equation}
which allows us to compute the modulus of conjugation by elements
of the torus
\begin{equation}\label{torus.action}
\left(\prod_{\beta\in\Sigma^\circ}h_\beta(b_\beta)\right)\prod_{\alpha\in\Sigma^-}x_{\alpha}(a_\alpha)
\left(\prod_{\beta\in\Sigma^\circ}h_\beta(b_\beta)\right)^{-1}=\prod_{\alpha\in\Sigma^-}x_{\alpha}\left(\prod_{\beta\in\Sigma^\circ}
b_\beta^{\langle \alpha, \beta\rangle}a_\alpha\right).
\end{equation}
The resulting modulus of conjugation by $\prod_{\beta \in
\Sigma^\circ}h(b_\beta) \in T$ (in terms of the affine
coordinates) is
\begin{equation}\label{modulus}
\left|\prod_{\alpha\in\Sigma^-}\prod_{\beta\in\Sigma^\circ}
b_\beta^{\langle
\alpha,\beta\rangle}\right|=\left|\prod_{\beta\in\Sigma^\circ}
b_\beta^{\langle \sum_{\alpha \in \Sigma^-} \alpha,
\beta\rangle}\right|=\left|\prod_{\beta\in\Sigma^\circ}
b_\beta^{-\langle \rho, \beta\rangle}\right|.
\end{equation}

Combining \eqref{Haar.on.U}, \eqref{Haar.on.T} and
\eqref{modulus}, the following defines a Haar integral on $G(K)$
$$\int_{K^{\Sigma^-} \times K^{\Sigma^\circ} \times K^{\Sigma^+}}  \prod_{\beta \in \Sigma^\circ}
|a_\beta|^{-\langle \rho,\beta\rangle-1} f(\Phi(\Psi(a)))~
d\lambda.$$

We show that the measure $\mu$ as defined in
\eqref{haar.integral.formula} is normalized such that
$\mu(G(\lri))=1$. Let $I(\lri)$ be the inverse image of a Borel
subgroup $B(\resfield) \subset G(\resfield)$ under the reduction
map $G(\lri) \to G(\resfield)$, i.e. an Iwahori subgroup of
$G(K)$. The inverse image of $I(\lri)$ under $\Phi \circ \Psi$ is
\[
J:=\{(a_1,...,a_r,b_1,..,b_\ell,c_1,...,c_r) : v(a_i) > 0, v(b_j)=0,
v(c_k) \ge 0 \} \subset K^d,
\]
see \cite[Theorem 22]{Steinberg} for more details.

The additive measure of $J$ is
$(1-q^{-1})^{|\Sigma^\circ|}q^{-|\Sigma^-|}$. It follows that
\[
\begin{split}
\mu(G(\lri))&=\mu(I(\lri))[G(\lri):I(\lri)]\\
&=\left(k_G(q)^{-1}\int_J 1 d\lambda\right)\cdot \frac{|G(\resfield)|}{|B(\resfield)|}\\
&=k_G(q)^{-1}(1-q^{-1})^{|\Sigma^\circ|}q^{-|\Sigma^-|}\frac{|G(\resfield)|}{(q-1)^{|\Sigma^\circ|}
q^{|\Sigma^+|}}\\
&=k_G(q)^{-1}\frac{|G(\resfield)|}{q^{\dim G}}=1,
\end{split}
\]
which completes the proof of the Theorem.
\end{proof}

\section{Counting conjugacy classes in quotients of Chevalley groups}\label{section-counting-cc}

Let $G$ be a Chevalley group. In this section we express the conjugacy class zeta function of $G(\lri)$ in
terms of a definable integral over $\lri^d$ and prove Theorem
\ref{theorem-cc}. As in the introduction, we let $c_m=c_m(G,\lri)$
be the number of conjugacy classes in $G(\lri, m)$ and we put
\begin{equation}
\zeta^{\cc}_{G(\lri)}(s)=\sum_{m=0}^{\infty} c_m~ q^{-m s}\ \ \ \
(s\in\C).
\end{equation}
We write $k_G(q)=|G(\resfield)|q^{-d}$, where $q=|\resfield|$, $p=\mathrm{char}({\resfield})$ and $d=d_G=\dim G$.

\begin{lemma}\label{lemma-index-dimension-formula}
Let $X$ be a smooth scheme defined over $\lri$. Then
for all $m\geq 1$
\[
|X(\lri/\maxideal^m)|=|X(\resfield)|q^{(m-1)\dim X}.
\]
In particular the formula holds for algebraic groups defined over $\lri$.

\end{lemma}

\begin{proof} This is well-known, but we sketch a proof for completeness.
A point in $X(\lri/\maxideal^m)$ corresponds to a section $\mathrm{Spec}(\lri/\maxideal^m)
 \to X \times _\lri \lri/\maxideal^m$.
We claim that for any section $f_0: \mathrm{Spec}(\resfield)
 \to X \times _\lri \resfield$ there are $q^{(m-1)\dim X}$ sections $f$ which make the following diagram commute
\[
\begin{matrix}
X \times_\lri \lri/\maxideal^m & \overset{i}{\hookleftarrow} & X \times_\lri \lri/\maxideal \\
f \uparrow ~ \downarrow & & f_0 \uparrow ~ \downarrow \\
\mathrm{Spec}(\lri/\maxideal^m) & \hookleftarrow & \mathrm{Spec}(\lri/\maxideal).
\end{matrix}
\]
Indeed, such a map $f$ should factor through the spectrum of the complete local ring at the image of $i \circ f_0$. By smoothness this ring is isomorphic to $\lri/\maxideal^m[[t_1,...,t_{\dim X}]]$. It follows that we need to count the number of maps $\lri/\maxideal^m[[t_1,...,t_{\dim X}]] \to \lri/\maxideal^m$ with fixed image modulo $\maxideal$. The number of such maps is $q^{(m-1)\dim X}$, and the assertion follows.

\end{proof}

\begin{prop}\label{proposition-counting-cc}
$$\zeta_{G(\lri)}^{\cc}(s)-1=\frac{k_G(q)(Z_{G(\lri)}(s-d)-1)}{1-q^{s-d}},$$
where $$Z_{G(\lri)}(s)=\int_{G(\lri)\times G(\lri)}
||\{(\mathbf{xy-yx})_{ij}\suchthat  1\leq i,j\leq d\} ||^sd\nu$$
and $\nu$ is the normalized Haar measure on $G(\lri)\times G(\lri)$. Here
$\mathbf{x}=(x_{ij})$ and $\mathbf{y}=(y_{ij})$ are $d\times d$
matrices of indeterminates. The norm in the integrand is $||\{z_i:
i\in I\}||:=\max_{i\in I}\{|z_i| \}$.
\end{prop}

\begin{proof}

We define a function  $w: G(\lri)\times G(\lri)\ra \N$  giving
the depth of a commutator in the congruence  filtration: put
$w(x,y)=\max\{m\in\N: x^{-1}y^{-1}xy\in G^m(\lri)\}$ if
$x^{-1}y^{-1}xy\neq 1$, and $\infty$ otherwise. It is
straightforward to show that $w(x,y)=\min_{1\leq i,j\leq
n}\{v(xy-yx)_{ij}\}$.

We will need the following classical result for a finite group $H$: denoting by  $cc(H)$ the number of conjugacy classes in $H$,
we have (cf. \cite{dScc}) that
\begin{equation}\label{nonburnside}
cc(H)=|H|^{-1}\sum_{x\in H} |C_H(x)|= |H|^{-1}|\{(x,y)\in H\times H:xy=yx  \}|.
\end{equation}

Applying \eqref{nonburnside} to $G(\lri,m)$ it follows that for
all $m$, $c_m=|G(\lri)/G^m(\lri)|^{-1}e_m,$ where

\[e_m:=|\{({x},{y})\in G( \lri,m)\times G(\lri,m):
{x}^{-1}y^{-1}xy=1\}|.\]

Put

\[
\begin{split}
&W_m =\{({x},{y})\in G(\lri)\times G(\lri): w({x},{y})\geq m \},\\
&Z'_{G(\lri)}(s)=\sum_{m=0}^\infty q^{-ms}\nu(W_m).
\end{split}
\]
Then
\begin{align*}Z_{G(\lri)}(s)& = \sum_{m=0}^\infty q^{-ms}\nu(W_m \smallsetminus W_{m+1})=
Z'_{G(\lri)}(s)(1-q^s)+q^s.
\end{align*}
It is straightforward to  show that $\nu(W_m)=\nu(G^m(\lri)\times
G^m(\lri))e_m$. An elementary computation, using
Lemma~\ref{lemma-index-dimension-formula}, then gives

\begin{align*}
Z'_{G(\lri)}(s)&=\sum_{m=0}^{\infty} \nu (G^m(\lri)\times
G^m(\lri)) e_m q^{-ms}\\
&= \sum_{m=0}^{\infty} |G(\lri)/G^m(\lri)|^{-1}c_mq^{-ms}\\
&=1+k_G(q)^{-1}(\zeta_{G(\lri)}^{\cc}(s+d)-1),
\end{align*}

since $\nu(G^m(\lri)\times G^m(\lri))=|G(\lri)/ G^m(\lri)|^{-2}$. Proposition~\ref{proposition-counting-cc} now follows.
\end{proof}

\subsection{Proof of Theorem \ref{theorem-cc}}
Proposition~\ref{proposition-counting-cc} states that
\[
\zeta_{G(\lri)}^{\cc}(s)-1=\frac{k_G(q)(Z_{G(\lri)}(s-d)-1)}{1-q^{s-d}}.
\]
Applying Theorem~\ref{theorem-coord2} to
\[
Z_{G(\lri)}(s-d)=\int_{G(\lri)\times G(\lri)}
||\{(\mathbf{xy-yx})_{ij}\suchthat  1\leq i,j\leq d\} ||^{s-d}d\nu
\]
we obtain an integral on $K^{2d}$ with respect to the additive
Haar measure. Note that this converges for all $s\in \C$ with
$\mathrm{Re}(s)\geq d$. Finally we apply Theorem
\ref{theorem-model-theory} to this integral and this
renders the required expression for
$\zeta_{G(\lri)}^{\cc}(s)$.  For the last statement note that the
equality $Z_{G(\lri)}(s)=Z_{G(\lri')}(s)$, for $s\in \C$ satisfying
$\text{Re}(s)>\sigma_0$, implies that for any $m$ the groups
$G(\lri,m)$ and $G(\lri',m)$ have the same number of conjugacy
classes.

\section{Dimensions of Hecke modules of intertwining operators}\label{section-intertwining}

\subsection{Induced representations and intertwining operators}

Let $H$ be a finite group and let $Q_1,Q_2$ be subgroups of $H$.
Let
\[
\rho_i=\mathbf{1}_{Q_i}^H=\mathbb{C}[H/Q_i]
\]
denote the induced representation obtained by inducing the trivial
representation of $Q_i$ to $H$. The space $\Hom_H(\rho_1,\rho_2)$,
which we refer to as the Hecke (intertwining) module, of $H$-maps
between the corresponding representation spaces, is a bi-module
with a left $\mathrm{End}_H(\rho_1)$-action and right
$\mathrm{End}_H(\rho_2)$-action. It admits an algebraic
description as $\oplus_{\sigma \in \mathrm{Irr}(H)}
\Mat(n_{1,\sigma},n_{2,\sigma};\C)$, where $n_{i,\sigma}$ is the
multiplicity of an irreducible representation $\sigma$ of $H$ in
$\rho_i$, and a geometric description as the space of $Q_2$-$Q_1$
bi-invariant functions $\C[Q_2 \backslash H /Q_1]$. The latter
identification is obtained by interpreting a
$Q_2$-$Q_1$-bi-invariant function $\varphi$ as a summation kernel,
thus giving a $H$-invariant map $T_\varphi$ between the
corresponding representations:
\[
[T_\varphi(f)](hQ_2)=\sum_{g\in
H/Q_1}\varphi(Q_2h^{-1}gQ_1)f(gQ_1), \quad f\in\mathbb{C}[H/Q_1].
\]

In particular,
\begin{equation}\label{dim.equality}
\dim \C[Q_2 \backslash H /Q_1] = \dim \Hom_H(\rho_1,\rho_2)=
\sum_\sigma n_{1,\sigma}n_{2,\sigma}.
\end{equation}

The following Lemma is an analogue of \eqref{nonburnside}.

\begin{lemma}\label{lemma-be}
Let $H$ be a finite group and let $Q_1,Q_2$ be subgroups of $H$.
Then
\begindm
|\{Q_1hQ_2\suchthat h\in H\}|=\frac{1}{|Q_1||Q_2|}|\{(x,y)\in
H\times Q_2 \suchthat xyx^{-1}\in Q_1\}|.
\enddm
\end{lemma}
\begin{proof}
\begin{align*}
|\{Q_1hQ_2\suchthat h\in H\}|&=|Q_1\backslash H/Q_2|\\
&=\frac{1}{|Q_2|}\sum_{x\in Q_1\backslash H}|\stab_{Q_2}(x)|\\
&=\frac{1}{|Q_1||Q_2|}|\{(x,y)\in H\times Q_2 \suchthat Q_1 x=Q_1 xy\}|\\
&=\frac{1}{|Q_1||Q_2|}|\{(x,y)\in H\times Q_2 \suchthat
xyx^{-1}\in  Q_1\}|.
\end{align*}
\end{proof}

\subsection{Zeta functions associated with induced representations of Chevalley groups}

Let $G$ be a Chevalley group and let $\Sigma$ denote the set of
roots. Let $S$ be a closed set of roots, i.e.\ $\alpha, \beta \in
S$ and $\alpha+\beta \in \Sigma$ implies $\alpha+\beta \in S$. Let
$P_S$ be the standard parabolic subgroup of $G$ associated with
$S$, namely, the group generated by $\frak X_\alpha$, $\alpha \in
S$ (see Section~\ref{measure-transfer}). Let $\rho^S_{\lri,m}$ be
the complex valued permutation representation of $G(\lri)$ induced
from the action of $G(\lri)$ on the finite set
$G(\lri)/P_S(\lri)G^m(\lri)\simeq G(\lri,m)/P_S(\lri,m)$. To study
the family of representations $\rho^S_{\lri,m}$ ($S$ a closed set
of roots) it is natural to look at the spaces of intertwining maps
from $\rho^{S_1}_{\lri,m}$ to $\rho^{S_2}_{\lri,m}$. We have
\[
\Hom_{G(\lri)}\left(\rho^{S_1}_{\lri,m},
\rho^{S_2}_{\lri,m}\right) \simeq \mathbb{C}[P_{S_2}(\lri,m)
\backslash G(\lri,m)/P_{S_1}(\lri,m)],
\]
and denoting $b_{\lri,m}^{S_1,S_2}:=\dim_{\mathbb{C}}
\Hom_{G(\lri)}\left(\rho^{S_1}_{\lri,m},
\rho^{S_2}_{\lri,m}\right)$ we can define the zeta function
\[
\zeta^{S_1,S_2}_{G(\lri)}(s)=\sum_{m=0}^\infty
b_{\lri,m}^{S_1,S_2}q^{-ms}.
\]

In order to express $\zeta^{S_1,S_2}_{G(\lri)}(s)$ in terms of a definable integral we use Lemma~\ref{lemma-be}. Letting
\[ e^{S_1,S_2}_{\lri,m}=|\{({ x},{ y})\in
G(\lri,m)\times G(\lri,m) \suchthat { y}\in P_{S_2}(\lri,m), {
xyx}^{-1}\in P_{S_1}(\lri,m)\}|,
\]
we get
\begin{equation}\label{bSS}
b^{S_1,S_2}_{\lri,m}(G)=\frac{e^{S_1,S_2}_{\lri,m}}{|P_{S_1}(\lri,m)||P_{S_2}(\lri,m)|},
\end{equation}
where
\begin{equation}\label{order.of.parabolic}
|P_{S}(\lri,m)|=q^{m \dim P_{S}}\kGq{P_S}, \quad \mathrm{for ~
all}~m\geq 1,
\end{equation}
by Lemma~\ref{lemma-index-dimension-formula}.
\medskip

Let $\lambda_{S}$ be the function $G(\lri) \to
\naturals\cup\{\infty\}$ given by
\begindm
\lambda_{S}({ x})=\begin{cases}
\infty&\textrm{if}\ { x}\in P_{S}(\lri)\\
\max\{m\in\naturals\suchthat { x}\in
P_{S}(\lri)G^m(\lri)\}&\textrm{if}\ { x}\notin P_{S}(\lri).
\end{cases}
\enddm
We next define $w:G(\lri)\times G(\lri) \to
\naturals\cup\{\infty\}$ by
\begindm
w({ x},{ y})=\min\{\lambda_{S_2}({ y}),\lambda_{S_1}({
xyx}^{-1})\}.
\enddm

\begin{lemma}\label{definability-of-lambda} The functions $\lambda_S$ and $w$ are
$\cL_{\vf}$-definable.
\end{lemma}

\begin{proof} The Chevalley group comes with a specific embedding into
$\GL_d$, where $d=\dim G$ and hence $z \in G^m(\lri)$ if and only
if $m \le \min\{v\left((z-I)_{ij}\right) \suchthat 1 \le i,j, \le
d\}$. It follows that the maximum $m$ such that $z \in G^m(\lri)$
is equal to $\min_{i,j} v\left((z-I)_{ij}\right)$. Therefore,
\[
\lambda_S(x)= \min \{v\left((y^{-1}x-I)_{ij}\right) \suchthat y
\in P_S(\lri), 1\le i,j \le d\}.
\]
As the function $y \mapsto \max_{i,j} |(y^{-1}x-I)_{ij}|$ is
continuous on a compact set the maximum is achieved on
$P_S(\lri)$. Thus $\min \{v\left((y^{-1}x-I)_{ij}\right) \suchthat
y \in P_S(\lri), 1\le i,j \le d\}$ exists for all $x \in G(\lri)$
and hence $\lambda_S(x)$ is well defined and is clearly definable
by a formula of the value group sort. The definability of $w$
follows as well.
\end{proof}

For an algebraic subgroup $H$ of $G$ we denote its dimension by
$d_H$ and we write $k_H(q)=|H(\resfield)|/q^{d_H}$.

\begin{prop}\label{proposition-GLn}
\begindm
\zeta^{S_1,S_2}_{G(\lri)}(s)-1=\frac{\tqalpha{S_1,S_2}\left(Z^{S_1,S_2}_{G(\lri)}
(s-2 d_G + d_{P_{S_1}} + d_{P_{S_2}})-1\right)}{1-q^{s-2 d_G +
d_{P_{S_1}} + d_{P_{S_2}}}},
\enddm
where
\begindm
Z^{S_1,S_2}_{G(\lri)}(s):=\int_{G(\lri)\times G(\lri)}q^{-w(x,y)s}
d\nu,
\enddm
$\nu$ is the normalised Haar  measure on $G(\lri)\times G(\lri)$
and $\tqalpha{S_1,S_2}=\kGq{G}^2\kGq{P_{S_1}}^{-1}
\kGq{P_{S_2}}^{-1}$.

\end{prop}

\begin{proof}

The proof is very similar to that of
Proposition~\ref{proposition-counting-cc}. Put
\[
\begin{split}
W_m &= \{({ x},{ y})\in G(\lri)\times G(\lri)\suchthat w({ x},{
y})\geq m\}, \\
 Z'(s)&=\sum_{m=0}^\infty
\nu(W_m)q^{-ms}.
\end{split}
\]
We have
\begin{align}
Z^{S_1,S_2}_{G(\lri)}(s)&=\int_{G(\lri)\times
G(\lri)}q^{-w(x,y)s} d\nu\nonumber\\
&=\sum_{m=0}^\infty \nu(W_m \smallsetminus
W_{m+1})q^{-ms}=Z'(s)(1-q^s)+q^s.\nonumber
\end{align}
It is readily shown that
$\nu(W_m)=e^{S_1,S_2}_{\lri,m}\nu(G^m(\lri)\times G^m(\lri))$, thus

\begin{align*}
Z'(s)&=\sum_{m=0}^\infty e^{S_1,S_2}_{\lri,m}\nu(G^m(\lri)\times G^m(\lri))q^{-ms}\\
&=1+\sum_{m=1}^\infty e^{S_1,S_2}_{\lri,m}\kGq{G}^{-2}q^{-2m\dim{G}-ms}\tag{by\ \text{Lemma}~\ref{lemma-index-dimension-formula}}\\
&=1+ \sum_{m=1}^\infty|P_{S_1}(\lri,m)||P_{S_2}(\lri,m)| b^{S_1,S_2}_{\lri,m} \kGq{G}^{-2}q^{-m(2\dim{G}+s)}\tag{by\ (\ref{bSS})}\\
&=1+ \sum_{m=1}^\infty \tqalpha{S_1,S_2}^{-1}  b_{\lri,m}^{S_1,S_2} q^{-m(2\dim{G}-\dim P_{S_1}-\dim P_{S_2}+s)}\tag{by\ (\ref{order.of.parabolic})}\\
&=1+\tqalpha{S_1,S_2}^{-1}\left(
\zeta^{S_1,S_2}_{G(\lri)}\left(s+2 d_G - d_{P_{S_1}} -
d_{P_{S_2}}\right)-1 \right),\nonumber
\end{align*}
which gives the desired result.
\end{proof}

\subsection{Proof of Theorem \ref{theorem-hecke}} Theorem \ref{theorem-hecke} follows by combining Proposition
\ref{proposition-GLn}, Lemma \ref{definability-of-lambda},
Theorem~\ref{theorem-coord2}, and Corollary \ref{cor-of-B} using
the convergence of $\zeta^{S_1,S_2}_{G(\lri)}(s)$ for all $s\in
\C$ with $\mathrm{Re}(s)\geq d_G$ for all $\lri$, in a similar way to the
proof of Theorem \ref{theorem-cc}.

\section{Some remarks}

\subsection{Extending the class of groups} We have not described all of the groups to which the results of this paper
apply. For example, if $G=\GL_n$ then using the embedding of $G$
in $\Mat_n$ one can apply Theorem \ref{theorem-model-theory}
directly and deduce that the conjugacy class zeta function depends
only on the residue field of $\lri$. Similarly, if $G$ is an
elliptic curve, then the coefficients of its conjugacy class zeta
function are merely the numbers of $\lri/\maxideal^m$-points of the
curve and depend only on the residue field. Similarly, the results
apply if $G$ is any standard parabolic or unipotent subgroup of a
Chevalley group.

\subsection{One zeta function for all primes} The main tool in
this paper is the transfer principle based on Model theory
(Theorem \ref{theorem-model-theory}), thereby a (possibly) large
set of primes is excluded. Nevertheless, in cases where explicit
calculations are tangible, it transpires that the zeta function is
much more rigid. If $G$ is either $\GL_2$ or $\GL_3$ then there
exists a rational function $R_G(X,Y) \in \mathbb{Q}(X,Y)$,
explicitly computed in \cite{AOPV}, such that
$\zeta^{\cc}_{G(\lri)}(s)=R_G(q,q^{-s})$ where $\lri$ is any
compact discrete valuation ring and $q=|\lri/\maxideal|$.

A similar phenomenon occurs for the zeta function of Hecke modules
of intertwining operators. If $G$ is either $\GL_2$ or $\GL_3$,
and $P,Q$ are any standard parabolics thereof, then there exists a
rational function $R_G^{P,Q}(X,Y) \in \mathbb{Q}(X,Y)$, the
coefficients of which are explicitly computed in \cite{OPV}, such
that $\zeta^{P,Q}_{G(\lri)}(s)=R_G^{P,Q}(q,q^{-s})$, where $\lri$
is any compact discrete valuation ring and $q=|\lri/\maxideal|$.

Another instance of this phenomenon can be found when $G$ is the
discrete Heisenberg group over $\lri$, that is, the group of upper
unitriangular matrices over $\lri$, with a residue field
$\resfield$ of cardinality $q$. The dimension of this group is 3,
and by Proposition~\ref{proposition-counting-cc},
$$\zeta_{G(\lri)}^{\cc}(s)=\frac{1+q^{3-s}Z_{G(\lri)}(s-3)}{1-q^{3-s}},$$
where
$$Z_{G(\lri)}(s)=\int_{\lri^4} |ab-cd|^s
d\mu=\frac{(1-q^{-1})(1-q^{-2})}{(1-q^{-1-s})(1-q^{-2-s})}.$$ Thus
$$\zeta_{G(\lri)}^{\cc}(s)=\frac{1+q^{-s}-q^{1-s}-q^{2-s}+q^{3-s}}{(1-q^{1-s})(1-q^{2-s})(1-q^{3-s})}.$$
The integral $Z_{G(\lri)}(s)$ is a classical example of Igusa's
zeta function, and the formula above can be found in various
places, for instance \cite{Igusa}.

\subsection{Euler products}

One can form an Euler product of the local conjugacy class zeta
functions to count conjugacy classes in a global object. Let $F$
be a global field, i.e.\ a number field or the field of rational
functions of a curve over a finite field. Let $\gri$ be its ring
of integers. For a finite place $v$ of $F$ let $F_v$ and $\gri_v$
be the completions with respect to $v$ of $F$ and $\gri$,
respectively. For an ideal $\mathcal{I} \lhd \gri$ denote
$N(\mathcal{I})=[\gri:\mathcal{I}]$. Let $G \le \GL_n$ be a
$\Z$-defined algebraic subgroup and let $c_{\mathcal{I}}$ stand
for the number of conjugacy classes in $G(\gri/\mathcal{I})$.
Define the (global) conjugacy zeta function to be
\[
\zeta_{G(\gri)}^{\cc}(s)=\sum_{\mathcal{I} \lhd \gri}
c_{\mathcal{I}} N(\mathcal{I})^{-s}.
\]
\begin{lemma} If $\gri$ and $G$ as above, and assuming $G$ has strong approximation, then
\[
\zeta_{G(\gri)}^{\cc}(s)=\prod_{\mathcal{P} \in
\mathrm{Spec}(\gri)}\zeta_{G(\gri_\mathcal{P})}^{\mathrm{cc}}(s).
\]
\end{lemma}
\begin{proof} By strong approximation
%\footnote{PP: I've now decided that strong approximation is what we need -- this is the same as Marcus uses in his conjugacy counting paper}
 we have an isomorphism
\[
G(\gri/\mathcal{I}) \isom G(\gri/\mathcal{P}_1^{r_1}) \times \dots
\times G(\gri/\mathcal{P}_k^{r_k}),
\]
where $\mathcal{I}=\mathcal{P}_1^{r_1}\dots\mathcal{P}_k^{r_k}$.
For the irreducible representations we have identifications
\[
\mathrm{Irr}\left(G(\gri/\mathcal{I})\right) =
\mathrm{Irr}\left(G(\gri/\mathcal{P}_1^{r_1})\right) \times \dots
\times \mathrm{Irr}\left(G(\gri/\mathcal{P}_k^{r_k})\right).
\]
The class numbers are therefore multiplicative and the result
follows.

\end{proof}

\subsection{Representation zeta functions} Let $G$ be a
representation rigid group, i.e. the number of $n$-dimensional
complex irreducible representations, denoted $r_n(G)$, is finite
for any $n\in\mathbb{N}$. Then
\[
\zeta^{\mathrm{rep}}_G(s):=\sum_{n=1}^{\infty}r_n(G)n^{-s}
\]
is the representation zeta function of $G$; see \cite{LL}. If $G$ is topological
then one counts continuous representations. We are tempted to pose
the following.\\

\begin{Questionnn} Let $G$ be a Chevalley group. Does there exist a
constant $N$ such that for any complete discrete valuation ring
$\lri$ with residue field of size $q$ and characteristic $p > N$,
the representation zeta function (of continuous representations)
$\zeta^{\mathrm{rep}}_{G(\lri)}(s)$ depends only on $q$ and not on $\lri$?
\end{Questionnn}

Hrushovski and Martin \cite{HM2007} have proved rationality results 
for Poincar\myacute{e} series which arise from the numbers of equivalence classes in 
definable families of equivalence relations on $\mathbb{Q}_p$ and 
count certain (twist-isoclasses of) representations of nilpotent groups. They use model-theoretic 
results on imaginaries in $\Q_p$ and in algebraically closed valued fields. It would be
interesting to know if their methods could be used to address uniformity in~$p$.

\end{document}